\documentclass[10pt,letterpaper]{amsart}

\newcommand{\la}{\langle}
\newcommand{\ra}{\rangle}
\renewcommand{\Re}{\operatorname{Re}}
\renewcommand{\Im}{\operatorname{Im}}

\usepackage{amssymb}
\usepackage{amsthm}
\usepackage{amsxtra}
\usepackage{graphicx}
\usepackage{mathabx}

\newtheorem{theorem}{Theorem}

\newtheorem{proposition}[theorem]{Proposition}
\newtheorem{lemma}[theorem]{Lemma}
\newtheorem{corollary}[theorem]{Corollary}

\theoremstyle{remark}
\newtheorem{remark}[theorem]{Remark}

\numberwithin{equation}{section}

\numberwithin{theorem}{section}

\numberwithin{table}{section}

\numberwithin{figure}{section}

\ifx\pdfoutput\undefined
  \DeclareGraphicsExtensions{.pstex, .eps}
\else
  \ifx\pdfoutput\relax
    \DeclareGraphicsExtensions{.pstex, .eps}
  \else
    \ifnum\pdfoutput>0
      \DeclareGraphicsExtensions{.pdf}
    \else
      \DeclareGraphicsExtensions{.pstex, .eps}
    \fi
  \fi
\fi

\title[LWP for NLS with Potentials]{Local-in-time Well-posedness for Nonlinear Schr\"odinger Equations with Potentials}

\date{\today}
\author{Younghun Hong}
\address{Brown University}
\begin{document}

\maketitle
\section{Introduction}
In this note, we prove the local-in-time well-posedness (LWP) and the mass and energy conservation laws for a 3d cubic nonlinear Schr\"odinger equation with a real-valued potential $V(x)$:
\begin{equation}\tag{$\textup{NLS}_V$}
iu_t+\Delta u-Vu\pm|u|^2u=0;\ u(0)=u_0\in H^s,
\end{equation}
by a contraction mapping argument. The main tools are Strichartz estimates (Lemma 3.4), derived from the dispersive estimate of Beceanu and Goldberg \cite{BG}, and the norm equivalence between the standard Sobolev norm and the Sobolev norm associated with a Schr\"odinger operator $\mathcal{H}=-\Delta+V$ (Lemma 3.2).

\subsection{Potential classes}
We say that a real-valued function $V$ is contained in \textit{Kato class} if
$$\lim_{r\to0+}\sup_{x\in\mathbb{R}^3}\int_{|x-y|\leq r}\frac{|V(y)|}{|x-y|}dy=0,$$
and it is in \textit{global Kato class} $\mathcal{K}$ if its \textit{global Kato norm} 
$$\|V\|_{\mathcal{K}}:=\sup_{x\in\mathbb{R}^3} \int_{\mathbb{R}^3}\frac{|V(y)|}{|x-y|}dy$$
is finite. Let $C_c^b$ be the set of bounded, compactly supported functions. Define the potential class $\mathcal{K}_0$ as the norm closure of $C_c^b$ within the \textit{global Kato norm}. The weak $L^{3/2}$-space is denoted by $L^{3/2,\infty}$.

\subsection{Resolvent} Let $\sigma(\mathcal{H})$ be the spectrum of $\mathcal{H}$. For $z\in\sigma(\mathcal{H})$, we define the resolvent operator by $R_V(z)=(\mathcal{H}-z)^{-1}$. Recall the free resolvent formula
\begin{equation}
(R_0(z)f)(x)=((-\Delta-z)^{-1}f)(x)=\int_{\mathbb{R}^3}\frac{e^{\sqrt{-z}|x-y|}}{4\pi|x-y|}f(y) dy,\ z\in\mathbb{C}\setminus[0,+\infty).
\end{equation}

\subsection{Acknowledgement} The author would like to thank Vedran Sohinger for explaining that \cite[Theorem 1.2.5]{Ca} is useful to show persistence of regularity.

\section{Spectral Properties}
We begin by reviewing the spectral properties of Schr\"odinger operators $\mathcal{H}=-\Delta+V$.  We will show that under suitable assumptions, $(a)$ $\mathcal{H}$ is self-adjoint on $L^2$ with domain $\mathfrak{h}^2$, which is dense in $H^1$; $(b)$ the spectrum $\sigma(\mathcal{H})$ is purely absolutely on the positive real-line $[0,+\infty)$ and has at most finitely many negative eigenvalues.

\begin{lemma}[Quadratic form \cite{DP}]
Let $V$ be in Kato class, and define the quadratic form by
$$q(u,v):=\int_{\mathbb{R}^3} u\overline{\mathcal{H}v}dx.$$
Then there exists $a\gg1$ such that
$$\frac{1}{2}\|\nabla u\|_{L^2}^2-a\|u\|_{L^2}^2\leq q(u,u)\leq \frac{3}{2}\|\nabla u\|_{L^2}^2+a\|u\|_{L^2}^2$$
for all $u\in H^1$. Thus, $q$ is lower semi-bounded with form domain $H^1$.
\end{lemma}

\begin{proof}
Let $a$ be a large number to be chosen later. We claim that 
\begin{equation}
\||V|^{\frac{1}{2}}(-\Delta+2a)^{-\frac{1}{2}} u\|_{L^2}\leq \frac{1}{2}\|u\|_{L^2}^2,
\end{equation}
which is, by the standard $TT^*$ argument with $T=|V|^{\frac{1}{2}}(-\Delta+2a)^{-\frac{1}{2}}$, equivalent to
\begin{equation}
\||V|^{\frac{1}{2}}R_0(-2a)|V|^{\frac{1}{2}}u\|_{L^2}\leq \frac{1}{2} \|u\|_{L^2}.
\end{equation}
Note that $(2.2)$ follows once we prove that 
\begin{equation}
\sup_{x\in\mathbb{R}^3}\int_{\mathbb{R}^3}\frac{|V(y)|e^{-\sqrt{2a}|x-y|}}{4\pi|x-y|} dy\leq\frac{1}{2},
\end{equation}
since by the free resolvent formula $(1.1)$ and the H\"older inequality, the left hand side of $(2.2)$ is less than 
\begin{align*}
&\Big\{\int_{\mathbb{R}^3} |V(x)|\Big|\int_{\mathbb{R}^3}\frac{e^{-\sqrt{2a}|x-y|}}{4\pi|x-y|}|V(y)|^{\frac{1}{2}}u(y) dy\Big|^2dx\Big\}^{\frac{1}{2}}\\
&\leq\Big\{\int_{\mathbb{R}^3} |V(x)|\Big(\int_{\mathbb{R}^3}\frac{|V(y)|e^{-\sqrt{2a}|x-y|}}{4\pi|x-y|} dy\Big)\Big(\int_{\mathbb{R}^3}\frac{e^{-\sqrt{2a}|x-y|}}{4\pi|x-y|}|u(y)|^2 dy\Big)dx\Big\}^{\frac{1}{2}}\\
&\leq\Big\{\frac{1}{2}\int_{\mathbb{R}_y^3} \Big(\int_{\mathbb{R}_x^3}\frac{|V(x)|e^{-\sqrt{2a}|x-y|}}{4\pi|x-y|}dx\Big)^2|u(y)|^2 dy\Big\}^{\frac{1}{2}}\leq\frac{1}{2}\|u\|_{L^2}.
\end{align*}
To show $(2.3)$, we choose $r_0>0$ such that
$$\sup_{x\in\mathbb{R}^3}\int_{|x-y|\leq r_0}\frac{|V(y)|e^{-\sqrt{2a}|x-y|}}{4\pi|x-y|} dy\leq\sup_{x\in\mathbb{R}^3}\int_{|x-y|\leq r_0}\frac{|V(y)|}{4\pi|x-y|} dy\leq\frac{1}{4},$$
in particular,
\begin{equation}
\int_{|y|\leq r_0}\frac{|V(y)|e^{-\sqrt{2a}|y|}}{4\pi|y|} dy\leq\frac{1}{4}.
\end{equation}
Let $I_k=B(0, 2^{k+1}r_0)\setminus B(0, 2^kr_0)$. Since each $I_k$ is covered by $O(2^k)$-many balls $B(x_j,r_0)$, we have
\begin{align*}
\int_{I_k}\frac{|V(y)|e^{-\sqrt{2a}|y|}}{4\pi|y|} dy&\leq e^{-\sqrt{2a} 2^kr_0}\int_{I_k}\frac{|V(y)|}{4\pi|y|} dy\leq e^{-\sqrt{2a} 2^kr_0}\sum_{j=1}^{O(2^k)}\int_{B(x_j,r_0)}\frac{|V(y)|}{4\pi|y|} dy\\
&\sim e^{-\sqrt{2a} 2^kr_0}\sum_{j=1}^{O(2^k)}\int_{B(x_j,r_0)}\frac{|V(y)|}{|x_j-y|} dy\lesssim e^{-\sqrt{2a}2^kr_0}2^k.
\end{align*}
Now we choose $a\gg1$ depending on $r_0$ so that 
\begin{equation}
\int_{|y|\geq r_0}\frac{|V(y)|e^{-\sqrt{2a}|y|}}{4\pi|y|} dy\leq\sum_{k=0}^\infty\int_{I_k}\frac{|V(y)|e^{-\sqrt{2a}|y|}}{4\pi|y|} dy\lesssim\sum_{k=0}^\infty e^{-\sqrt{2a}2^kr_0}2^k\leq\frac{1}{4}.
\end{equation}
Combining $(2.4)$ and $(2.5)$, we obtain $(2.3)$.
\end{proof}

\begin{remark}
Replacing $(2.3)$ by
$$\sup_{x\in\mathbb{R}^3}\int_{\mathbb{R}^3}\frac{|V(y)|}{4\pi|x-y|} dy\leq\frac{\|V\|_{\mathcal{K}}}{4\pi},$$
in the proof of Lemma 2.1, one can show that 
$$\int_{\mathbb{R}^3}V|u|^2dx\leq\frac{\|V\|_{\mathcal{K}}}{4\pi}\|\nabla u\|_{L^2}^2.$$
Hence, if the negative part of a potential, denoted by $V_-$, is small, precisely, $\|V_-\|_{\mathcal{K}}\leq4\pi$, then the quadratic form $q$ is positive-definite:
$$q(u,u)=\int_{\mathbb{R}^3} u\overline{\mathcal{H}u}dx\geq \int_{\mathbb{R}^3}|\nabla u|^2 dx+\int_{\mathbb{R}^3}V_-|u|^2dx\geq\Big(1-\frac{\|V_-\|_{\mathcal{K}}}{4\pi}\Big)\|\nabla u\|_{L^2}^2\geq0.$$
\end{remark}

\begin{lemma}[Self-adjointness]
If $V$ is in Kato class, then $\mathcal{H}$ is self-adjoint on $L^2$ with domain $\mathfrak{H}^2:=(\mathcal{H}+a+1)^{-1}(L^2)$.
\end{lemma}

\begin{proof}
Following \cite[Theorem VIII.15]{ReSi1}, we aim to show that $(1+a+\mathcal{H}):H^1\to H^{-1}$ is a quasi-isometry. Indeed, if it is true, the inverse $(1+a+\mathcal{H})^{-1}:H^{-1}\to H^1$ exists, and $(1+a+\mathcal{H})^{-1}: L^2\to \mathfrak{H}^2\subset L^2$ is self-adjoint. Hence, by functional calculus, $((1+a+\mathcal{H})^{-1})^{-1}=(1+a+\mathcal{H})$ is self-adjoint with domain $\mathfrak{H}^2$, and so is $\mathcal{H}$.

By $(2.1)$ and the $T^*T$ argument, $\|(1+a+\mathcal{H})u\|_{H^{-1}}\sim\|u\|_{H^1}$. For surjectivity, we introduce the norm $\|u\|_1:=(q(u,u)+(a+1)\|u\|_{L^2}^2)^{1/2}$, which is, by $(2.1)$, comparable to $\|u\|_{H^1}$. Polarizing $\|\cdot\|_1$, we get a new $H^1$-inner product $(u,v)_1:=q(u,v)+(a+1)\int_{\mathbb{R}^3}u\bar{v}dx$. For $f\in H^{-1}$, define a linear functional $\ell: H^1\to\mathbb{C}$ by $\ell(v)=\int_{\mathbb{R}^3} v\overline{f} dx$. Then it follows from the Riesz lemma with inner product $(\cdot,\cdot)_1$ that there exists $u\in H^1$ such that
$$\ell_f(v)=\int_{\mathbb{R}^3} v\overline{f} dx=(v,u)_1=\int_{\mathbb{R}^3}v\overline{(1+a+\mathcal{H})u}dx$$
for all $v\in H^1$, which implies $(1+a+\mathcal{H})u=f$.
\end{proof}

\begin{remark}
$(i)$ Since $(1+a+\mathcal{H})^{-1}$ is a quasi-isometry and $L^2$ is dense in $H^{-1}$, $\mathfrak{H}^2$ is a dense subset of $H^1$. In general, $\mathfrak{H}^2\neq H^2$ \cite{Sh}. $(ii)$ $\mathfrak{H}^2$ is a Hilbert space with inner product $(u,v)_2=\int_{\mathbb{R}^3} (1+a+\mathcal{H})u\overline{(1+a+\mathcal{H})v}dx$.
\end{remark}

\begin{lemma}[Essential spectrum]
If $V\in\mathcal{K}_0$, then $\sigma_\textup{ess}(\mathcal{H})=[0,+\infty)$.
\end{lemma}

\begin{proof}
Let $\lambda=-(1+a)$. Since $\mathcal{H}$ is semi-bounded, by Weyl's essential spectrum theorem \cite[Theorem XIII.14]{ReSi2}, it suffices to check that $R_V(\lambda)-R_0(\lambda)$ is compact on $L^2$. We denote $T_{V_1,V_2}(\lambda):=V_1R_0(\lambda)V_2$, where $V_1=|V|^{1/2}\textup{sign} (V)$ and $V_2=|V|^{1/2}$. Indeed, by $(2.2)$,
$$R_V(\lambda)-R_0(\lambda)=R_0(\lambda)((I+VR_0(\lambda))^{-1}-I)=-\sum_{n=0}^\infty R_0(\lambda)V_2(-T_{V_1,V_2}(\lambda))^n V_1R_0(\lambda)$$
is bounded on $L^2$. It remains to show that $V_1R_0(\lambda)$ is compact on $L^2$. For $\epsilon>0$, choose $V_\epsilon\in C_c^b$ such that $\|V-V_\epsilon\|_{\mathcal{K}}\leq \epsilon$ and $\textup{sign}(V)=\textup{sign}(V_\epsilon)$. By Remark 2.2, we have
$$\|(V_1-V_{\epsilon, 1})R_0(\lambda)f\|_{L^2}\leq\frac{\|V-V_\epsilon\|_{\mathcal{K}}}{4\pi}\|(-\Delta-\lambda)^{-\frac{1}{2}}f\|_{L^2}\lesssim\epsilon\|f\|_{L^2}.$$
Moreover, since $V_{\epsilon,1}$ is compactly supported, it follows from the Rellich compactness theorem that $V_{\epsilon,1}R_0(\lambda)$ is compact on $L^2$. Since $\epsilon$ is arbitrary, we conclude that $V_1R_0(\lambda)$ is compact.
\end{proof}

We call $z\in\mathbb{C}$ a \textit{resonance} of $\mathcal{H}$ if $(I+VR_0(z))$ is not invertible $\mathcal{L}(L^1)$.

\begin{lemma}[Negative eigenvalues] If $V\in\mathcal{K}_0$ and zero is not a resonance, then $\mathcal{H}$ has at most finitely many negative eigenvalues. 
\end{lemma}

\begin{proof}
Since eigenvalues are discrete, it suffices to exclude the scenario that negative eigenvalues are accumulated to zero. First, we claim that negative eigenvalues are resonances. Let $\psi\in \mathfrak{H}^2$ be an eigenfunction corresponding to a negative eigenvalue $\lambda$. For $\epsilon>0$, we choose $V_\epsilon \in C_c^b$ such that $\|V-V_\epsilon\|_{\mathcal{K}}<\epsilon$ and thus $\|(V-V_\epsilon)R_0(\lambda)|\|_{L^1\to L^1}\leq\frac{\epsilon}{4\pi}$. Set
$$\varphi:=-(I+(V-V_\epsilon)R_0(\lambda))^{-1}V_\epsilon R_0(\lambda)V\psi,$$
which is formally equivalent to 
$$(I+(V-V_\epsilon)R_0(\lambda))\varphi=-V_\epsilon R_0(\lambda)V\psi\Leftrightarrow(I+VR_0(\lambda))\varphi=0.$$
Thus,
\begin{align*}
\|\varphi\|_{L^1}&\leq\|(I+(V-V_\epsilon)R_0(\lambda))^{-1}V_\epsilon R_0(\lambda)V\psi\|_{L^1}\lesssim\|V_\epsilon R_0(\lambda)V\psi\|_{L^1}\\
&\lesssim\|V_\epsilon\|_{L^2} \|R_0(\lambda)V\psi\|_{L^2}\lesssim\|R_0(\lambda)|V|^{1/2}\textup{sign}(V)\|_{L^2\to L^2}\||V|(-\Delta)^{-\frac{1}{2}}(-\Delta)^{\frac{1}{2}}\psi\|_{L^2}\\
&\lesssim\||V|(-\Delta)^{-\frac{1}{2}}\|_{L^2\to L^2}\|\psi\|_{H^1}\lesssim\|\psi\|_{H^1}.
\end{align*}
By the claim, it is enough to show that resonances cannot be accumulated to zero. By the assumption, $(I+VR_0(0))$ is invertible in $\mathcal{L}(L^1)$.  We write 
\begin{align*}
(I+VR_0(\lambda))&=(I+VR_0(0)+V(R_0(\lambda)-R_0(0)))\\
&=(I+VR_0(0))(I+(I+VR_0(0))^{-1}V(R_0(\lambda)-R_0(0))).
\end{align*}
By the mean value theorem, we have
$$\|V_\epsilon (R_0(\lambda)-R_0(0))\|_{\mathcal{L}(L^1)}\leq\sup_{y\in\mathbb{R}^3}\Big\|V_\epsilon(x)\frac{e^{-\sqrt{-\lambda}|x-y|}-1}{4\pi|x-y|}\Big\|_{L_x^1}\lesssim \sqrt{-\lambda}\|V_\epsilon\|_{L^1}\to 0\textup{ as }\lambda\to 0-.$$
Moreover, we have
$$\|(V-V_\epsilon)(R_0(\lambda)-R_0(0))\|_{\mathcal{L}(L^1)}\lesssim\epsilon.$$
Hence, if $\lambda<0$ is sufficiently close to zero, one can make
$$\|(I+VR_0(0))^{-1}V(R_0(\lambda)-R_0(0))\|_{\mathcal{L}(L^1)}$$  
arbitrarily small. Thus, $(I+VR_0(\lambda))$ is invertible, and $\lambda$ is thus not a resonance.
\end{proof}

\begin{lemma}[Spectrum on the positive real-line \cite{BG}] If $V\in\mathcal{K}_0$ and $\mathcal{H}$ has no resonance on $[0,+\infty)$, then $\mathcal{H}$ has purely absolutely continuous spectrum on $[0,+\infty)$.
\end{lemma}

\section{Norm Equivalence and Strichartz Estimates}

We present two main items for LWP: norm equivalence and Strichartz estimates.

\begin{lemma}[Sobolev inequality] $(i)$ (inhomogeneous) If $V\in\mathcal{K}_0$, then there exists $a\gg1$ such that $\|(1+a+\mathcal{H})^{-\frac{s}{2}}u\|_{L^q}\lesssim\|u\|_{L^p}$, where $1<p\leq q<\infty$, $0<s<3$ and $\frac{1}{q}\geq\frac{1}{p}-\frac{s}{3}$.\\
$(ii)$ (homogeneous) If $V\in\mathcal{K}_0$ and $\|V_-\|_{\mathcal{K}}<4\pi$, then $\|\mathcal{H}^{-\frac{s}{2}}u\|_{L^q}\lesssim\|u\|_{L^p}$, where $1<p<q<\infty$, $0<s<3$ and $\frac{1}{q}=\frac{1}{p}-\frac{s}{3}$.
\end{lemma}

\begin{proof}
Choose $a\gg1$ such that $\|(a+V)_-\|_{\mathcal{K}}<4\pi$. Observe that by the Feynmann-Kac formula (A.26) of \cite{Si} and \cite[Theorem 2]{Ta}, there exist $A_1, A_2>0$ such that
$$0\leq e^{-t(1+a+\mathcal{H})}(x,y)=e^{t(\Delta-(1+a+V))}(x,y)\leq e^{-t}e^{t(\Delta-(a+V)_-)}(x,y)\leq \frac{A_1}{t^{3/2}}e^{-t-A_2\frac{|x-y|^2}{t}}.$$
Applying it to
\begin{equation}
(1+a+\mathcal{H})^{-\frac{s}{2}}=\frac{1}{\Gamma(s/2)}\int_0^\infty e^{-t(1+a+\mathcal{H})}t^{\frac{s}{2}-1}dt,
\end{equation}
we obtain the kernel estimate for $(1+a+\mathcal{H})^{-\frac{s}{2}}$. The Sobolev inequality then follows from the inhomogeneous fractional integration inequality. By the same way, we can show $(ii)$.
\end{proof}

We denote the standard inhomogeneous (homogeneous, resp) Sobolev space by $W^{s,r}$ ($\dot{W}^{s,r}$, resp). We define the inhomogeneous (homogeneous, resp) Sobolev norm associated with $\mathcal{cH}$ by
$$\|u\|_{\mathfrak{W}^{s,r}}:=\|(1+a+\mathcal{H})^{\frac{s}{2}}u\|_{L^r}\ (\|u\|_{\dot{\mathfrak{W}}^{s,r}}:=\|\mathcal{H}^{\frac{s}{2}}u\|_{L^r})$$
where $a$ is a large number given by Lemma 3.1. We denote $\|u\|_{\mathfrak{H}^s}:=\|u\|_{\mathfrak{W}^{s,2}}$ and $\|u\|_{\dot{\mathfrak{H}}^s}:=\|u\|_{\dot{\mathfrak{W}}^{s,2}}$. The following lemma says that these two norms are equivalent for some $r$.

\begin{lemma}[Norm equivalence] If $V\in\mathcal{K}_0\cap L^{3/2,\infty}$, then 
$$\|u\|_{\mathfrak{W}^{s,r}}\sim\|u\|_{W^{s,r}}\textup{ and }\|u\|_{\dot{\mathfrak{W}}^{s,r}}\sim\|u\|_{\dot{W}^{s,r}},$$
where $1<r<\frac{3}{s}$ and $0\leq s\leq2$.
\end{lemma}

\begin{proof}
Following \cite{DFVV}, we will show the inhomogeneous case. We omit the proof for the homogeneous case, since it can be proved similarly.

By the H\"older inequality and the Sobolev inequality in the Lorentz norms, we have
$$\|(1+a-\Delta)u\|_{L^r}\leq\|f\|_{\mathfrak{W}^{2,r}}+\|Vf\|_{L^r}\lesssim\|f\|_{\mathfrak{W}^{2,r}}+\|V\|_{L^{3/2,\infty}}\|f\|_{L^{\frac{3r}{3-2r},r}}\lesssim\|f\|_{\mathfrak{W}^{2,r}}$$
for $1<r<\frac{3}{2}$. Moreover, since $e^{-t(1+a+\mathcal{H})}$ and $e^{-t(1+a-\Delta)}$ satisfy the Gaussian heat kernel estimate \cite{Ta}, it follows from Sikora and Wright \cite{SW} that the imaginary power operators $(1+a+\mathcal{H})^{iy}$ and $(1+a-\Delta)^{iy}$, with $y\in\mathbb{R}$, are bounded on $L^r$ for all $1<r<\infty$ with the operator norm $\sim\la y\ra^{3/2}$. Therefore, we obtain 
\begin{align*}
\|(1+a-\Delta)^{iy}(1+a+\mathcal{H})^{-iy}\|_{L^r\to L^r}&\lesssim\la y\ra^3\textup{ for }1<r<\infty\textup{ and }y\in\mathbb{R}\\
\|(1+a-\Delta)^{2+iy}(1+a+\mathcal{H})^{-2-iy}\|_{L^r\to L^r}&\lesssim\la y\ra^3\textup{ for }1<r<\tfrac{3}{2}\textup{ and }y\in\mathbb{R}.
\end{align*}
By the Stein-Weiss complex interpolation theorem as in \cite{DFVV}, we conclude that 
$$\|f\|_{W^{s,r}}\sim\|(1+a-\Delta)^{\frac{s}{2}}f\|_{L^r}\lesssim\|(1+a+\mathcal{H})^{\frac{s}{2}}f\|_{L^r}=\|u\|_{\mathfrak{W}^{s,r}}\textup{ for }1<r<\tfrac{3}{s}.$$
The other direction follows from the same argument.
\end{proof}

Recall the dispersive estimate of Beceanu and Goldberg \cite{BG}:
\begin{lemma}[Dispersion estimate \cite{BG}] If $V\in\mathcal{K}_0$ and $\mathcal{H}$ has no resonance on $[0,+\infty)$, then
$$\|e^{-it\mathcal{H}}P_c\|_{L^1\to L^\infty}\lesssim |t|^{-3/2},$$
where $P_c$ is the spectral projection to the continuous spectrum $[0,+\infty)$.
\end{lemma}
By the argument of \cite{KT}, we derive Strichartz estimates from the dispersive estimate. For $0\leq s<\frac{3}{2}$, an exponent pair $(q,r)$ is called $s$-\textit{admissible} if $2\leq q\leq\infty$ and $2\leq r<3/s$ and $\frac{2}{q}+\frac{3}{r}=\frac{3}{2}$. We define the inhomogeneous (homogeneous) distorted Strichartz norm by
$$\|u\|_{\mathfrak{S}^s(I)}:=\sup_{(q,r):\ H^s\text{-admissible}} \|u\|_{L^q_{t\in I}\mathfrak{W}^{s,r}_x}\ \Big(\|u\|_{\dot{\mathfrak{S}}^s(I)}:=\sup_{(q,r):\ H^s\text{-admissible}} \|u\|_{L^q_{t\in I}\dot{\mathfrak{W}}^{s,r}_x}\Big).$$
Similarly, we define the inhomogeneous (homogeneous) standard Strichartz norm by
$$\|u\|_{S^s(I)}:=\sup_{(q,r):\ H^s\text{-admissible}} \|u\|_{L^q_{t\in I}\mathfrak{W}^{s,r}_x}\ \Big(\|u\|_{\dot{S}^s(I)}:=\sup_{(q,r):\ H^s\text{-admissible}} \|u\|_{L^q_{t\in I}\dot{\mathfrak{W}}^{s,r}_x}\Big).$$
Note that by the norm equivalence, $\|u\|_{\mathfrak{S}^s(I)}\sim \|u\|_{S^s(I)}$ and $\|u\|_{\dot{\mathfrak{S}}^s(I)}\sim\|u\|_{\dot{S}^s(I)}$.

\begin{lemma}[Strichartz estimates] If $V\in\mathcal{K}_0$ and $\mathcal{H}$ has no resonance on $[0,+\infty)$, then
\begin{align*}
\|e^{-it\mathcal{H}}P_cf\|_{\mathfrak{S}^s(I)}&\lesssim\|f\|_{\mathfrak{H}^s},\\
\Big\|\int_0^t e^{-i(t-s)\mathcal{H}}P_cF(s)ds\Big\|_{\mathfrak{S}^s(I)}&\lesssim\|F\|_{L_t^2 \mathfrak{W}_x^{s,6/5}}.
\end{align*}
In both inequalities, the inhomogeneous norms can be replaced by the homogeneous ones.
\end{lemma}

\section{Local-in-time Well-posedness and Energy Conservation Law}

Now we prove local well-posedness (LWP), persistence of regularity and conservation laws for  nonlinear Schr\"odinger equations with potentials. We begin by LWP in $H^s$ for $s\in(\frac{1}{2},1]$.

\begin{theorem}[LWP in $H^s$ for $\textup{$s\in(\frac{1}{2},1]$}$]
If $V\in \mathcal{K}_0\cap L^{3/2,\infty}$ and $\mathcal{H}$ has no resonance on $[0,+\infty)$, then $\textup{NLS}_V$ is locally-in-time well-posed in $H^s$ for $s\in(\frac{1}{2},1]$.
\end{theorem}

\begin{proof}
For $u_0\in \mathfrak{H}^s$, let $\mathfrak{X}:=\{u: \|u\|_{\mathfrak{S}^s(I)}\leq  2c\|u_0\|_{\mathfrak{H}^s}\}$, where $c$ is a large constant and $I=[0,T]$ with $T\in(0,1]$ to be chosen later. By the norm equivalence $\|u\|_{\mathfrak{H}^s}\sim \|u\|_{H^s}$, the theorem follows once we show that
\begin{equation}
\begin{aligned}
\Phi_{u_0}(u):&=e^{-it\mathcal{H}}u_0\pm i\int_0^t e^{-i(t-s)\mathcal{H}}(|u|^2u)(s)ds\\
&=e^{-it\mathcal{H}}P_cu_0\pm i\int_0^t e^{-i(t-s)\mathcal{H}}P_c(|u|^2u)(s)ds+\sum_{j=1}^Je^{-it\lambda_j}\la u_0,\psi_j\ra_{L^2}\psi_j\\
&\pm i\sum_{j=1}^J\la\int_0^t e^{-i(t-s)\lambda_j}(|u|^2u)(s)ds, \psi_j\ra_{L^2}\psi_j
\end{aligned}
\end{equation}
is a contraction on $\mathfrak{X}$, where $\psi_j\in\mathfrak{H}^2$ is a normalized eigenfunction such that $\|\psi_j\|_{L^2}=1$. Note that by the Sobolev inequality and the norm equivalence, $\|\psi_j\|_{\mathfrak{S}^s(I)}\lesssim\|\psi_j\|_{\mathfrak{H}^2}$ and $\|\psi_j\|_{H^{\frac{3}{2}-}}\sim\|\psi_j\|_{\mathfrak{H}^{\frac{3}{2}-}}\leq\|\psi_j\|_{\mathfrak{H}^2}$. For notational convenience, we omit the time interval $I$ in the norm $\|\cdot\|_{L_{t\in I}^r}$ unless there is no confusion.

Applying the Strichartz estimates to $(4.1)$, we write
\begin{align*}
\|\Phi_{u_0}(u)\|_{\mathfrak{S}^s(I)}&\lesssim\|u_0\|_{\mathfrak{H}^s}+\||u|^2u\|_{L_t^2\mathfrak{W}_x^{s,6/5}}+\|u_0\|_{L^2}\sum_{j=1}^J\|\psi_j\|_{L^2}\|\psi_j\|_{\mathfrak{S}^s(I)}\\
&+\||u|^2u\|_{L_t^1H_x^{-3/2+}}\sum_{j=1}^J\|\psi_j\|_{H^{\frac{3}{2}-}}\|\psi_j\|_{\mathfrak{S}^s(I)}\\
&\lesssim\|u_0\|_{\mathfrak{H}^s}+\||u|^2u\|_{L_t^2\mathfrak{W}_x^{s,6/5}}+\||u|^2u\|_{L_t^1L_x^{1+}}
\end{align*}
By the norm equivalence, the fractional Leibniz rule and the Sobolev inequality,
\begin{align*}
\||u|^2u\|_{L_t^2\mathfrak{W}_x^{s,6/5}}&\sim\||u|^2u\|_{L_t^2W_x^{s,6/5}}\lesssim \|u\|_{L_t^4 L_x^6}^2\|u\|_{L_t^\infty H_x^s}\\
&\lesssim T^{\frac{2s-1}{2}}\|u\|_{L_t^{\frac{2}{1-s}} \mathfrak{W}_x^{s,\frac{6}{1+2s}}}^2\|u\|_{L_t^\infty \mathfrak{H}_x^s}\leq T^{\frac{2s-1}{2}}\|u\|_{\mathfrak{S}^s(I)}^3,\\
\||u|^2u\|_{L_t^1L_x^{1+}}&=T\|u\|_{L_t^\infty L_x^{3+}}^3\lesssim T\|u\|_{L_t^\infty \mathfrak{H}_x^s}^3\leq T\|u\|_{\mathfrak{S}^s(I)}^3.
\end{align*}
Therefore, for $u\in\mathfrak{X}$, we obtain 
\begin{equation}
\|\Phi_{u_0}(u)\|_{\mathfrak{S}^s(I)}\leq c\|u_0\|_{\mathfrak{H}^s}+cT^{\frac{2s-1}{2}}\|u\|_{\mathfrak{S}^s(I)}^3\leq c\|u_0\|_{\mathfrak{H}^s}+cT^{\frac{2s-1}{2}}(2c\|u_0\|_{\mathfrak{H}^s})^3.
\end{equation}
Similarly, we estimate the difference
\begin{align*}
\Phi_{u_0}(u)-\Phi_{u_0}(v)&=\pm i\int_0^t e^{-i(t-s)\mathcal{H}}P_c(|u|^2u-|v|^2v)(s)ds\\
&\pm i\sum_{j=1}^J\la\int_0^t e^{-i(t-s)\lambda_j}(|u|^2u-|v|^2v)(s)ds, \psi_j\ra_{L^2}\psi_j
\end{align*}
to get
\begin{equation}
\begin{aligned}
\|\Phi_{u_0}(u)-\Phi_{u_0}(v)\|_{\mathfrak{S}^s(I)}&\leq cT^{\frac{2s-1}{2}}(\|u\|_{\mathfrak{S}^s(I)}^2+\|v\|_{\mathfrak{S}^s(I)}^2)\|u-v\|_{\mathfrak{S}^s(I)}\\
&\leq cT^{\frac{2s-1}{2}}(2c\|u_0\|_{\mathfrak{H}^s})^2\|u-v\|_{\mathfrak{S}^s(I)}.
\end{aligned}
\end{equation}
From $(4.2)$ and $(4.3)$, we conclude that if $T$ is small enough depending on $\|u_0\|_{\mathfrak{H}^s}\sim\|u_0\|_{H^s}$, then $\Phi_{u_0}$ is a well-defined contraction mapping on $\mathfrak{X}$.
\end{proof}

Next, we prove the mass and energy conservation laws:
$$M[u(t)]=\|u(t)\|_{L^2}^2=M[u_0]\textup{ (mass)},$$
$$E[u(t)]=\int_{\mathbb{R}^3} \frac{1}{2}|\nabla u(t)|^2+\frac{1}{2}V|u(t)|^2dx\mp\frac{1}{4}|u(t)|^4 dx=E[u_0]\textup{ (energy)}.$$ 
Formally, if $u$ solves $\textup{NLS}_V$, then 
\begin{equation}
\begin{aligned}
\frac{d}{dt}M[u]&=2\Re\int_{\mathbb{R}^3}u_t\bar{u}dx=\Im\int_{\mathbb{R}^3}(\mathcal{H}u\mp|u|^2u)\bar{u}dx=0,\\
\frac{d}{dt}E[u]&=\Re\int_{\mathbb{R}^3}u_t\overline{(\mathcal{H}u\mp|u|^2u)}dx=\Im\int_{\mathbb{R}^3}(\mathcal{H}u\mp|u|^2u)\overline{(\mathcal{H}u\mp|u|^2u)}dx=0.
\end{aligned}
\end{equation}
However, we need some extra works to make the above calculation rigorous, since $\mathcal{H}u$ is not defined for $H^1$-solutions.

\begin{lemma}[LWP in $\mathfrak{H}^2$] If $V\in \mathcal{K}_0\cap L^{3/2,\infty}$ and $\mathcal{H}$ has no resonance on $[0,+\infty)$, then $\textup{NLS}_V$ is locally well-posed in $\mathfrak{H}^2$. Here, the maximal existence time has lower bound depending only on the $H^s$-norm of initial data for $\frac{3}{4}<s\leq 1$.
\end{lemma}

For the proof, we need the following lemma.
\begin{lemma} In the situation of Lemma 4.2, $\tilde{\mathfrak{X}}_{a,b}(I)=\{u: \|u\|_{L_{t\in I}^\infty \mathfrak{H}_x^2}\leq a\textup{ and }\|u\|_{\mathfrak{S}^s(I)}\leq  b\}$ is complete with respect to the $\mathfrak{S}^s(I)$-norm, where $I$ is a finite interval.
\end{lemma}

\begin{proof}
Recall \cite[Theorem 1.2.5]{Ca}: ``Consider two Banach spaces $X\hookrightarrow Y$ and $1<p,q\leq\infty$. Let $\{f_n\}_{n=1}^\infty$ be a bounded sequence in $L^q(Y)$ and let $f: I\to Y$ be such that $f_n(t)\rightharpoonup f(t)$ in $Y$ as $n\to \infty$ for a.e. $t\in I$. If $\{f_n\}_{n=1}^\infty$ is bounded in $L^p(I;X)$ and if $X$ is reflexive, then $f\in L^p(I; X)$ and $\|f\|_{L^p(I; X)}\leq\underset{n\to\infty}\liminf\|f_n\|_{L^p(I; X)}$." Fix an admissible pair $(q,r)$, and let $X=\mathfrak{H}^2$ and $Y=\mathfrak{W}^{s,r}$. Then, by the Sobolev inequality and Remark 2.4, $X$ and $Y$ satisfies the assumptions in \cite[Theorem 1.2.5]{Ca}. Let $\{u_n\}_{n=1}^\infty$ be a Cauchy sequence in $\tilde{\mathfrak{X}}_{a,b}(I)$. Then, $u_n\to u$ in $L_{t\in I}^q \mathfrak{W}_x^{s,r}$. Hence, it follows from \cite[Theorem 1.2.5]{Ca} that $\|u\|_{L_{t\in I}^\infty \mathfrak{H}_x^2}\leq \underset{n\to\infty}\liminf\|u_n\|_{L_{t\in I}^\infty \mathfrak{H}_x^2}\leq a$, and thus $u\in\tilde{\mathfrak{X}}_{a,b}(I)$. 
\end{proof}

\begin{proof}[Proof of Lemma 4.2]
We only give a sketch of the proof, since it is similar to that of Theorem 4.1. Let $\epsilon=\frac{4s-3}{10}$ $(\Rightarrow0<\epsilon\leq\frac{1}{10})$. For $u_0\in \mathfrak{H}^2$, let $\tilde{\mathfrak{X}}=\tilde{\mathfrak{X}}_{a,b}(I)$ with $a=2\|u_0\|_{\mathfrak{H}^2}$ and $b=2\|u_0\|_{\mathfrak{H}^s}$, where $I=[0,T]$ is a short time interval to be chosen later. We want to show that $\Phi_{u_0}(u)$, given by $(4.1)$, is a contraction on $\tilde{\mathfrak{X}}$.

Consider
$$\int_0^t e^{-i(t-s)\mathcal{H}}P_c(|u|^2u)(s)ds.$$
First, applying the Strichartz estimates, we get 
$$\Big\|\int_0^t e^{-i(t-s)\mathcal{H}}P_c(|u|^2u)(s)ds\Big\|_{L_t^\infty \mathfrak{H}_x^2}\lesssim\|(1+a+\mathcal{H})(|u|^2u)\|_{L_t^2 L_x^{6/5}}.$$
Since $(1+a+\mathcal{H})(|u|^2u)$ is bounded by
\begin{align*}
&|(1+a-\Delta)(|u|^2u)|+|V||u|^3\\
&\lesssim |u|^2|(1+a-\Delta)u|+V|u|^3+|\nabla u|^2|u|+|u|^3\\
&\lesssim |u|^2|(1+a+\mathcal{H})u|+V|u|^3+|\nabla u|^2|u|+|u|^3,
\end{align*}
we write
\begin{align*}
\Big\|\int_0^t e^{-i(t-s)\mathcal{H}}P_c(|u|^2u)(s)ds\Big\|_{L_t^\infty \mathfrak{H}_x^2}&\leq\|u\|_{L_t^4L_x^6}^2\|u\|_{L_t^\infty\mathfrak{H}_x^2}+\|V\|_{L^{3/2,\infty}}\||u|^3\|_{L_t^2L_x^{6,6/5}}\\
&+\||\nabla u|^2|u|\|_{L_t^2 L_x^{6/5}}+\||u|^3\|_{L_t^2 L_x^{6/5}}=A+B+C+D.
\end{align*}
We estimate each term by the Sobolev inequality, the norm equivalence and the definition of the $\mathfrak{S}^s(I)$-norm:
\begin{align*}
A&\lesssim \|u\|_{L_t^4\mathfrak{W}_x^{\frac{1}{2}+,3-}}^2\|u\|_{L_t^\infty\mathfrak{H}_x^2}\lesssim T^{0+}\|u\|_{L_t^{4+}\mathfrak{W}_x^{\frac{1}{2}+,3-}}^2\|u\|_{L_t^\infty\mathfrak{H}_x^2}\leq T^{0+}\|u\|_{\mathfrak{S}^s(I)}^2\|u\|_{\mathfrak{S}^2(I)},\\
B&\lesssim\|u\|_{L_t^4L_x^{\frac{12}{1-6\epsilon}}}^2\|u\|_{L_t^\infty L_x^{\frac{1}{\epsilon},\frac{3}{2+3\epsilon}}}\lesssim T^\epsilon\|u\|_{L_t^{\frac{4}{1-2\epsilon}}\mathfrak{W}_x^{\frac{3+10\epsilon}{4},\frac{3}{1+\epsilon}}}^2\|u\|_{L_t^\infty \mathfrak{H}_x^2}\leq T^\epsilon\|u\|_{\mathfrak{S}^s(I)}^2\|u\|_{\mathfrak{S}^2(I)},\\
C&\leq \|u\|_{L_t^4 L_x^{12}}\||\nabla u|^2\|_{L_t^4 L_x^{4/3}}\lesssim\|u\|_{L_t^4 W_x^{\frac{3}{4}+, 3-}}\|u\|_{L_t^4 W_x^{\frac{3}{4}+,3}}\|u\|_{L_t^\infty W_x^{\frac{5}{4}-,\frac{12}{5}}}\\
&\lesssim T^{0+}\|u\|_{L_t^{4+} \mathfrak{W}_x^{\frac{3}{4}+, 3-}}\|u\|_{L_t^4 \mathfrak{W}_x^{\frac{3}{4}+,3}}\|u\|_{L_t^\infty \mathfrak{W}_x^{\frac{5}{4}-,\frac{12}{5}}}\leq T^{0+}\|u\|_{\mathfrak{S}^s(I)}^2\|u\|_{\mathfrak{S}^2(I)},\\
D&\leq\|u\|_{L_t^6 L_x^3}^2\|u\|_{L_t^6 L_x^6}\lesssim\|u\|_{L_t^6 \mathfrak{H}_x^s}^2T^{\frac{1}{6}}\|u\|_{L_t^\infty \mathfrak{H}_x^2}\leq T^{\frac{1}{6}}\|u\|_{\mathfrak{S}^s(I)}^2\|u\|_{\mathfrak{S}^2(I)}.
\end{align*}
In the second inequality for $C$, we used the fractional integration by parts:
$$\|fg\|_{L^p}\lesssim\||\nabla|^s f\|_{L^{q_1}}\||\nabla|^{-s}g\|_{L^{r_1}}+\||\nabla|^{-s}f\|_{L^{q_2}}\||\nabla|^sg\|_{L^{r_2}}$$
where $\frac{1}{p}=\frac{1}{q_1}+\frac{1}{r_1}=\frac{1}{q_2}+\frac{1}{r_2}$. This can be proved by a modification of the proof of the fractional Leibniz rule, see \cite{ChWe} for example. Thus, we get 
$$\Big\|\int_0^t e^{-i(t-s)\mathcal{H}}P_c(|u|^2u)(s)ds\Big\|_{L_t^\infty \mathfrak{H}_x^2}\lesssim T^{0+}\|u\|_{\mathfrak{S}^s(I)}^2\|u\|_{\mathfrak{S}^2(I)}\lesssim T^{0+}\|u_0\|_{\mathfrak{H}^s}^2\|u_0\|_{\mathfrak{H}^2}.$$
We estimate other terms in $\Phi_{u_0}$ as we did in Theorem 4.1. Collecting all, we prove that 
$$\|\Phi_{u_0}(u)\|_{L_t^\infty \mathfrak{H}_x^2}\leq\|u_0\|_{\mathfrak{H}_x^2}+cT^{0+} \|u_0\|_{\mathfrak{H}^s}^2\|u_0\|_{\mathfrak{H}_x^2}\textup{ for all }u\in\tilde{\mathfrak{X}}.$$
By choosing small $T$ depending only on $\|u_0\|_{\mathfrak{H}^s}\sim\|u_0\|_{H^s}$ together with $(4.2)$ and $(4.3)$, we conclude that $\Phi_{u_0}$ is contractive on $\tilde{\mathfrak{X}}$.
\end{proof}

\begin{corollary}[Persistence of regularity] In the situation of Lemma 4.2, we have the following properties: for a sequence of initial data $\{u_{n,0}\}_{n=1}^\infty\subset\mathfrak{H}^2$ with $u_{n,0}\to u_0$ in $H^s$, we denote by $\{u_n\}_{n=1}^\infty$ and $u$ corresponding solutions. Then, $\|u_n-u\|_{C_{t\in I}^0 H_x^s}\to 0$, where $I$ is a uniform interval given by Lemma 4.2, depending only on $\|u_0\|_{H^s}$.
\end{corollary}

\begin{proof}
By Lemma 4.2, there exists a uniform time interval $I$ on which all $u_n$ exists in $\mathfrak{H}_x^2$. Then the corollary follows from the $H^s$-continuity of the data-to-solution map in Theorem 4.1.
\end{proof}

\begin{proposition}[Conservation laws] In the situation of Theorem 4.1 with $s=1$, solutions conserve the mass and the energy.
\end{proposition}

\begin{proof}
We only show the energy conservation. The mass conservation can be proved similarly. Since the energy is continuous in $H^1$, by persistence of regularity, it suffices to verify that $\mathfrak{H}^2$-solutions conserve the energy. Suppose that $u(t)\in C_{t\in I}^0\mathfrak{H}_x^2$ solves the integral equation
$$u(t)=e^{-it\mathcal{H}}u_0\pm i\int_0^t e^{-i(t-s)\mathcal{H}}(|u|^2u)(s)ds.$$
Then $u(t)$ is time-differentiable in $L^2$. Indeed, by the estimates in the proof of Theorem 4.1 and Lemma 4.2, one can show that
$$\|(|u|^2u)(t)\|_{C_{t\in I}L_x^2}<\infty,\ \Big\|\int_0^t e^{is\mathcal{H}}(|u|^2u)(s)ds\Big\|_{\mathfrak{H}_x^2}<\infty.$$
Therefore the nonlinear term is differentiable in $L^2$:
\begin{align*}
&\frac{1}{\epsilon}\Big[\int_0^{t+\epsilon}e^{-i(t+\epsilon-s)\mathcal{H}}(|u|^2u)(s)ds-\int_0^te^{-i(t-s)\mathcal{H}}(|u|^2u)(s)ds\Big]\\
&-(|u|^2u)(t)+i\mathcal{H}\int_0^te^{-i(t-s)\mathcal{H}}(|u|^2u)(s)ds\\
&=\Big[\frac{1}{\epsilon}\int_t^{t+\epsilon}e^{-i(t+\epsilon-s)\mathcal{H}}(|u|^2u)(s)ds-(|u|^2u)(t)\Big]\\
&+\Big[\frac{1}{\epsilon}(e^{-i(t+\epsilon)\mathcal{H}}-e^{-it\mathcal{H}})+i\mathcal{H}e^{-it\mathcal{H}}\Big]\Big(\int_0^t e^{is\mathcal{H}}(|u|^2u)(s)ds\Big)\to 0\textup{ in }L^2\textup{ as }\epsilon\to0.
\end{align*}
Moreover, it is obvious that the linear term $e^{-it\mathcal{H}}u_0$ is differentiable. Thus, $u$ solves $iu_t-\mathcal{H}u\pm|u|^2u=0$ in $L_{t\in I}^\infty L_x^2$ and the formal calculation $(4.4)$ makes sense.
\end{proof}

Consider a 3d defocusing cubic nonlinear Schr\"odinger equation with a potential:
\begin{equation}\tag{$\textup{NLS}_V^-$}
iu_t+\Delta u-Vu-|u|^2u=0;\ u(0)=u_0.
\end{equation}
Recall that when $V=0$, the energy conservation law yields the global-in-time well-posedness (GWP) in $H^1$. We prove that the same is true in the presence of a potential:
\begin{theorem}[GWP] If $V\in \mathcal{K}_0\cap L^{3/2,\infty}$ and $\mathcal{H}$ has no resonance on $[0,+\infty)$, then $\textup{NLS}_V^-$ is globally-in-time well-posed in $H^1$.
\end{theorem}

\begin{proof}
Let $u(t)$ be a solution. By the norm equivalence and conservation laws,
\begin{align*}
\|u(t)\|_{H^1}^2&\sim\|(1+a+\mathcal{H})^{1/2}u(t)\|_{L^2}^2=\|\nabla u(t)\|_{L^2}^2+\int_{\mathbb{R}^3}V|u(t)|^2dx+(1+a)\|u(t)\|_{L^2}^2\\
&\leq 2 E[u(t)]+(1+a)M[u(t)]=2 E[u_0]+(1+a)M[u_0]<\infty
\end{align*}
for during its existence time. Thus $u(t)$ is global in $H^1$.
\end{proof}

\appendix

\end{document}